\documentclass[10pt]{article}
\usepackage{amsfonts,amsmath,amsthm,amssymb,easybmat,graphicx}
\usepackage{fullpage}
\usepackage{color}

\usepackage{fancyhdr}

\numberwithin{equation}{section}

\newtheorem*{question*}{Question}

\newtheorem{thm}{Theorem}[section]
\newtheorem{prop}[thm]{Proposition}

\newtheorem{cor}[thm]{Corollary}

\newtheorem{defin}[thm]{Definition}

\newtheorem{lemma}[thm]{Lemma}

\newcommand{\N}{\mathfrak N}
\newcommand{\Z}{\mathfrak Z}
\newcommand{\V}{\mathcal V}

\newcommand{\g}{\mathfrak g}
\newcommand{\h}{\mathfrak h}
\newcommand{\lk}{\mathfrak k}
\newcommand{\p}{\mathfrak p}
\newcommand{\vpq}{V_{pq}}
\newcommand{\rpqc}{\mathbb R^{p+q}(C)}

\newcommand{\J}{ \begin{BMAT}{cc}{cc}\ 0&1\\ -1&0\end{BMAT}} 

\newcommand{\restrictto}[2]{\left. #1 \right|_{#2}}

\begin{document}
\title{Moduli of Einstein and Non-Einstein Nilradicals}
\author{M. Jablonski}
\date{}
\maketitle

The subject of left-invariant Ricci soliton metrics on nilpotent Lie groups has enjoyed quite a bit of attention in the past several years.  These metrics are intimately related to left-invariant Einstein metrics on non-unimodular solvable Lie groups.  In fact, a classification of one is equivalent to a classification of the other.  In this note, we focus our attention on nilpotent Lie groups and Lie algebras.  We refer the reader to \cite{Heber,LauretStandard}, and references therein,  for more information about the connection between solvable and nilpotent groups with said metrics.  If a nilpotent Lie group admits a left-invariant Ricci soliton metric, then it is called  an \textit{Einstein nilradical};  otherwise it is called a \textit{non-Einstein nilradical}.

In this note we are concerned with the following question.

\begin{question*} How are the Einstein and non-Einstein nilradicals distributed among nilpotent Lie algebras?
\end{question*}

A full answer to the above question is not known.  However, if we restrict our attention to smaller classes of nilpotent Lie algebras, then more can be said.  For example, a generic two-step nilpotent Lie algebra is an Einstein nilradical (see Theorem \ref{thm: generic two-step are einstein}).  As the moduli space of isomorphism classes of nilpotent algebras is not a Hausdorff space, care must be taken in formulating this result.  This is discussed in the sequel.\\

In this note we state the known results for the case of generic two-step nilalgebras and give some new results on the set of non-generic two-step nilalgebras.  For example, we construct some new continuous families  of both Einstein and non-Einstein nilradicals (see Theorems \ref{thm: moduli of einstein nilradicals} \& \ref{thm: adjoin several extra tuples}).  The construction of such families has received much attention in the literature recently; see, e.g., \cite{Heber,GordonKerr:NewHomogeneousEinstein,Will,Payne:ExistenceofSolitononNil,Kerr:ADeformationOfQuaternionic,
Eberlein:prescribedRicciTensor, LauretWill:EinsteinSolvExistandNonexist,  Jablo:Thesis, Nikolayevsky:EinsteinSolvmanifoldsattchedtoTwostepNilradicals, Will:CurveOfNonEinsteinNilradicals}.

One of the families of Einstein nilradicals constructed here has larger dimension than that of the generic set,   see Corollary \ref{cor: einstein nilrads of type (2,2k+1)}.  This demonstrates the delicate nature of studying the moduli space of isomorphism classes of algebras and trying to speak about `generic' algebras.

We finish the note by giving a construction of several families of non-Einstein nilradicals.  These families are non-trivial in the sense that they consist of indecomposable algebras; that is, they do not arise as a direct sum of ideals (cf. Theorem \ref{thm: Einstein and direct sums}).  Moreover, the dimension of these families can be made arbitrarily large (depending on the dimension of the underlying vector space), see Theorems \ref{thm: non-Einstein for p=2-6} \& \ref{thm: adjoin several extra tuples}.  To demonstrate the extent of our results we present the following.  Recall that a two-step nilpotent Lie algebra is said to be of type $(p,q)$ if the commutator has dimension $p$ and codimension $q$.\\

\noindent \textbf{Corollary \ref{thm: summary of indecomp. of types (p,q)}}  For $(p,q)$ satisfying $8\leq q$ and $2\leq p \leq \frac{5}{4}q-8$, there exist indecomposable non-Einstein two-step nilradicals of type $(p,q)$.  Moreover, most of these types admit moduli of such algebras.  The dimension of said moduli is bounded below by $\frac{1}{8}[q-\frac{4}{5}(p+8)-7]$.\\

While this lower bound is crude, it is most easily presented.  We summarize our results graphically in Figure 2.\\

\textit{Acknowledgements.}   The work contained in this note is an extension of my thesis work done under the direction of Patrick Eberlein.  I would like to thank him for encouraging me to think about interesting problems.  I would also like to thank Jorge Lauret for some interesting conversations.

\section{Two-step nilpotent Lie algebras and $\rpqc$}\label{section: two-step nilpotent lie algebras and RpqC}

\subsection*{Representations and moment maps}
We begin with some brief general information on representations of reductive groups.  Let $G$ be a real reductive algebraic group.  Let $V$ be a real vector space on which $G$ acts linearly and rationally.  We can endow $V$ with an inner product $\langle , \rangle$ so that $G$ is self-adjoint, that is, so that $G$ is closed under the transpose operation relative to $\langle , \rangle$.  This provides us with a so-called Cartan decomposition of $G$ and $\g$; these decompositions are $G=KP$ and $\g = \lk \oplus \p$.  Here $K = G\cap O(\langle , \rangle)$, where $O(\langle , \rangle)$ is the orthogonal group of $\langle , \rangle$, and $\lk = LK$.  The endomorphisms of $\lk$ are skew-symmetric relative to $\langle , \rangle$.  The subspace $\p = symm \cap \g$, where $symm$ denotes the symmetric endomorphisms relative to $\langle , \rangle$, and $P = exp(\p)$.  For more information about Cartan decompositions see \cite{Mostow:SelfAdjointGroups}, \cite{RichSlow}, or \cite{JabloDistinguishedOrbits}.

We may also endow $\g$ with an inner product $\langle \langle , \rangle \rangle$ so that $Ad(K)$ acts orthogonally and $ad(\p)$ acts by symmetric transformations.  In the sequel our group of interest will be semi-simple and we choose ${\langle \langle X,Y\rangle \rangle = -B(X,Y^t)}$ where $B$ is the Killing form of $\g$.

Given these structures, we may construct a $\g$-valued moment map for the representation $G\times V \to V$.  The moment map is a homogeneous polynomial $m: V\to \p$ defined implicitly by
    $$ \langle \langle m(v),X\rangle \rangle = \langle X\cdot v,v \rangle$$
This map is $K$-equivariant where $K$ acts on $\p$ via the Adjoint representation.  The moment map has been used to study the orbit structure of representations of complex reductive groups in \cite{Kirwan,Ness,Jablo:FinitenessTheorem-compatibleSubgroups} and real reductive groups in \cite{RichSlow,Marian,EberleinJablo,JabloDistinguishedOrbits}.

\begin{defin}\label{def: distinguished point}  A point $v\in V$ is called distinguished if $m(v)\cdot v = rv$ for some $r\in \mathbb R$.  An orbit $G\cdot v$ is called distinguished if it contains a distinguished point.
\end{defin}

\textit{Remark.}  The distinguished points are precisely the critical points of the induced polynomial $||m||^2$ on projective space $\mathbb PV$.  We observe that if $v$ is a point such that $m(v) = 0$, then $v$ is distinguished.  These special distinguished points are the so-called minimal vectors for the representation, see \cite{JabloDistinguishedOrbits} for more information on distinguished points and orbits.

\subsection*{Moduli and detecting $G$-orbits along subvarieties}

Here we state the main results from \cite{Jablo:FinitenessTheorem-compatibleSubgroups}.  In that work general techniques are given to determine whether an orbit is distinguished and also to count the moduli of orbits intersecting subvarieties.  We will not give any proofs and refer the reader to that article for proofs and more information.

Let $G$ be a reductive algebraic group and $H$ a reductive subgroup.  Let $G$ act linearly and rationally on $V$ and suppose that $W$ is an $H$-stable subspace.  The vector space $V$ is assumed to be endowed with an inner product $\langle , \rangle$ such that both $G$ and $H$ are closed under the metric adjoint or transpose.  This is always possible, see \cite{Mostow:SelfAdjointGroups} or \cite{Jablo:FinitenessTheorem-compatibleSubgroups}.   We point out that the results of \cite{Jablo:FinitenessTheorem-compatibleSubgroups} hold more generally  for  subvarieties which are smooth and have smooth projections in projective space, however we only need to apply the results to linear spaces here.

\begin{defin}\label{def: H detected along W} We say that $G$ is $H$-detected along $W$ if $m_G(w)\in \h $ for $w\in W$.  Here $m_G$ denotes the moment map for the $G$ action on $V$.
\end{defin}

Being $H$-detected along $W$ is equivalent to $m_G(w)=m_H(w)$ for all $w\in W$.

\begin{thm}\label{thm: finiteness thm and detecting orbits via H}  Suppose that $W$ is an $H$-stable subspace along which  $G$ is $H$-detected.  Then for $w\in W$,  the following are true
\begin{enumerate}
    \item $G\cdot w \cap W$ is a finite union whose components are $H^0$-orbits, where $H^0$ is the identity component of $H$
    \item $G\cdot w$ is distinguished if and only if $H \cdot w$ is distinguished
\end{enumerate}
\end{thm}

This theorem is crucial in counting the moduli of isomorphism classes of algebras with certain natural symmetries.  See Sections \ref{section 2} \& \ref{section 3}.

\subsection*{Two-step nilpotent algebras as points in a representation}

Let $N$ be a Lie group with Lie algebra $\N$.  We will denote the bracket of $\N$ by $[\cdot , \cdot ]$.  The group $N$, or the algebra $\N$, is said to be two-step nilpotent if $ [[\N,\N],\N] = 0$.  This is equivalent to the condition  $[\N , \N] \subset \Z$, where $\Z$ is the center of $\N$.  A two-step nilpotent Lie algebra is said to be of type $(p,q)$ if $dim \ [\N , \N] =p$ and $codim \  [\N ,\N ] =q$.

The Lie algebra structure, the bracket, is completely determined by its values on a basis.  Let $\N$ be a two-step nilalgebra of type $(p,q)$ and let $\mathcal B = \{ v_1, \dots ,v_q, Z_1,\dots , Z_p\}$ be a basis of $\N$ such that $\{Z_1,\dots ,Z_p\}$ is a basis of $[\N,\N]$, the commutator of $\N$.  Such a basis is called an \textit{adapted basis} of $\N$.  Consider the following tuple $C=(C^1,\dots, C^p) \in \mathfrak{so}(q,\mathbb R)^p$ defined by
    $$  [v_i,v_j] = \sum_k C_{ij}^k Z_k$$
The skew-symmetry of each matrix $C^k$ follows from the fact that the Lie bracket is anti-symmetric.  In this way we can associate to each adapted basis $\mathcal B$ a tuple of skew-symmetric matrices $C_\mathcal B$.  The condition that $\N$ be of type $(p,q)$ is equivalent to the condition that the $C^i$ be linearly independent in $\mathfrak{so}(q,\mathbb R)$.  The set of $C\in \mathfrak{so}(q,\mathbb R)$ whose coordinates are linearly independent forms a non-empty, Zariski open set which we denote by $\vpq^0$.  Observe that this imposes the condition $p\leq D_q:=\frac{1}{2}q(q-1)$.  When $q$ is understood, sometimes we write $D=D_q$.

Consider $\mathbb R^n$, where $n=p+q$, and the usual basis $\{e_1,\dots, e_q,e_{q+1},\dots,e_{q+p}\}$.  Let $C\in \vpq^0$ and construct a two-step nilpotent Lie bracket on $\mathbb R^n$ as follows
    $$[e_i,e_j] = \sum_k C^k_{ij} e_{q+k}$$
We denote this Lie algebra by $\rpqc$.  Every two-step nilpotent Lie algebra is isomorphic to some $\rpqc$, see the theorem below.

Ultimately we are interested in the left-invariant geometry of Lie groups.  A Lie group with a left-invariant metric is equivalent to a Lie algebra endowed with an inner product.  We will denote a Lie algebra $\N$ with inner product $\langle , \rangle$ by the pair $\{ \N, \langle , \rangle\}$.  We will study such Lie algebras.  The Lie algebra $\rpqc$ will be given the inner product so that the usual basis $\{e_i\}$ is orthonormal.

As we will be interested in isomorphism classes and changes of basis, we consider the following action of $G = GL(q,\mathbb R)\times GL(p,\mathbb R)$ on $\vpq = \mathfrak{so}(q,\mathbb R)^p$.  Consider the isomorphism $\mathfrak{so}(q,\mathbb R)^p = \mathfrak{so}(q,\mathbb R)\otimes \mathbb R^p$ where $(0,\dots, 0, C^i,0,\dots,0) \leftrightarrow C^i\otimes e_i$; we define our action via this identification.  For $M\in \mathfrak{so}(q,\mathbb R)$, $v\in \mathbb R^p$, $g\in GL(q,\mathbb R)$, and $h\in GL(p,\mathbb R)$ we have
    $$(g,h)\cdot M\otimes v = (gMg^t) \otimes (hv)$$
where $GL(p,\mathbb R)$ acts on $\mathbb R^p$ by the standard representation and we extend the action of $GL(q,
\mathbb R)\times GL(p,\mathbb R)$ linearly to all vectors of $\vpq$.  This is just the tensor of the representations of $GL(q,\mathbb R)$ on $\mathfrak{so}(q,\mathbb R)$ and of $GL(p,\mathbb R)$ on $\mathbb R^p$.  We point out for later use that $K = O(q,\mathbb R)\times O(p,\mathbb R)$ is a maximal compact subgroup of $G = GL(q,\mathbb R)\times GL(p,\mathbb R)$.

\begin{thm}[Eberlein]\label{thm: Eberlein - isom classes}  Let $\{\N,\langle , \rangle\}$ be a metric two-step nilpotent Lie algebra of type $(p,q)$.  Then $\{\N,\langle , \rangle\}$ is isometric to $\rpqc$ for some $C\in \vpq^0$.  Moreover, the isomorphism class of $\rpqc$ comprises the orbit $GL(q,\mathbb R)\times GL(p,\mathbb R) \cdot C$ in $\vpq$ while the isometry class of $\rpqc$ comprises the orbit $O(q,\mathbb R)\times O(p,\mathbb R) \cdot C$ in $\vpq$.
\end{thm}

\textit{Remark}.  The fact that the quotient $\vpq^0 / G$ is not Hausdorff is one of the challenges  of talking about moduli of isomorphism classes of algebras.  However, in the sequel we demonstrate techniques to aid in measuring the size of some moduli of isomorphism classes.  For a proof of this theorem see \cite{Eberlein:prescribedRicciTensor} or \cite{Jablo:Thesis}.\\

As we will need the moment map for the above representation later, we record it here.  Let $m_1$ denote the moment map for the action of $GL(q,\mathbb R)$ on $\vpq$ and let $m_2$ denote the moment map for the action of $GL(p,\mathbb R)$ on $\vpq$.  Then the moment map of $G=GL(q,\mathbb R)\times GL(p,\mathbb R)$ is $m = (m_1,m_2)$ where
    \begin{eqnarray*} m_1(C) &=& -2\sum_i (C^i)^2 \\
        m_2(C)_{ij} &=& -tr(C^iC^j) \end{eqnarray*}
for $C = (C^1,\dots , C^p)$.  We point out that a point $C \in \vpq$ is $GL(q,\mathbb R)\times GL(p,\mathbb R)$-distinguished if and only if it is $SL(q,\mathbb R)\times SL(p,\mathbb R)$-distinguished.

\subsection*{Relating left-invariant geometry of nilpotent Lie groups with Geometric Invariant Theory}

Consider a nilpotent Lie group $N$ with Lie algebra $\N$.  A left-invariant metric on $N$ is equivalent to an inner product $\langle , \rangle$ on $\N$.  We denote the pair by $\{ \N, \langle , \rangle \}$.  We abuse notation and denote the inner product on $\N$ and the left-invariant metric on $N$ both by $\langle , \rangle$.

\begin{defin}  A left-invariant metric $\langle , \rangle$ on $N$ is called a Ricci soliton, or nilsoliton, if $Ric = \lambda Id + D$ for some $\lambda \in \mathbb R$ and some symmetric derivation $D$.  Here $Ric$ denotes the $(1,1)$ Ricci tensor relative to $\langle , \rangle$.
\end{defin}

For more information on nilsolitons we refer the reader to \cite{Lauret:CanonicalCompatibleMetric}.

\begin{thm}[Eberlein]\label{thm: Einstein nilradical iff distinguished orbit} Consider the metric two-step nilpotent Lie algebra $\rpqc$.  Then $\rpqc$ is a left-invariant Ricci soliton if and only if $C$ is a distinguished point of the representation $G = GL(q,\mathbb R)\times GL(p,\mathbb R)$ on $\vpq = \mathfrak{so}(q,\mathbb R)^p$.  Thus, $\rpqc$ is an Einstein nilradical if and only if the orbit $G \cdot C$ is distinguished.
\end{thm}

\textit{Remark}.  Using representations and distinguished points/orbits to study nilsolitons goes back to J. Lauret \cite{Lauret:CanonicalCompatibleMetric}  where general $k$-step nilpotent Lie groups are studied (see Definition \ref{def: distinguished point} for the definition of distinguished points and orbits).  The results above for two-step nilalgebras are not an immediate consequence of the known results for $k$-step nilalgebras.  We refer the reader to \cite{Eberlein:prescribedRicciTensor} or \cite{Jablo:Thesis} for a proof of the above theorem. (One could derive the above theorem from Lauret's work by applying the main results of \cite{Jablo:FinitenessTheorem-compatibleSubgroups}.)

The benefit of using these families of representations to study the case of two-step nilalgebras is that we can obtain information about `generic' algebras by looking at Zariski open sets in $\vpq$.  The following theorem is obtained by studying the generic orbits of $GL(q,\mathbb R)\times GL(p,\mathbb R)$ acting on $\vpq$.

\begin{thm}\label{thm: generic two-step are einstein}  There exists a Zariski open set $\mathcal O \subset \vpq$ such that the $GL(q,\mathbb R)\times GL(p,\mathbb R)$-orbit through any point $C\in \mathcal O$ has a distinguished orbit (cf. Definition \ref{def: distinguished point}).  Intersecting the set $\mathcal O$ with $\vpq^0$ shows that a generic two-step nilalgebra is a Einstein nilradical.  Moreover, the dimension of the moduli of isomorphism classes around a generic algebra is given below.  This dimension is also the dimension of nilsoliton metrics up to isometry and scaling, computed about a generic nilsoliton.
\end{thm}

\begin{center}
\begin{tabular}{|l|c|}
\multicolumn{2}{c}{
Dimension of Moduli
about generic points}\\
\hline

$(p,q)\ and \ (D-p,q)$ & dimension = $\mathcal M_{pq}$ \\
\hline

$(1,q)$ & 0 \\
$(2,4)$ & 0\\
$(2,2k),\ k\geq 3$ & k-3 \\
$(2,2k+1)$ & 0 \\
$(3,4)$ & 0\\
$(3,5)$ & 0\\
$(3,6)$ & 2\\
$(D,q)$ & 0\\
all other $(p,q)$& $p\ \frac{1}{2}q(q-1)-(q^2+p^2-2)-1$ \\ \hline
\end{tabular}
\end{center}

\textit{Remark.}  For a proof of the above theorem see \cite[Chapter 7]{Jablo:Thesis}.  The cases of $(p,q) \not = (1,q),(2,q), (D-1,q), (D-2,q)$ also appear in \cite{Eberlein:prescribedRicciTensor}.  All of the information needed to compute this is contained in the lists of Elashvili \cite{ElashviliStationarySubalgebra}.  Additionally this information was computed by Knop-Littlemann in \cite{KL}.  Note, the dimension of moduli will be the same for $(p,q)$ and the dual $(D-p,q)$.

In the remaining sections we will see that the non-generic algebras can have some very interesting behavior.  Most notable is that in many types $(p,q)$ one can construct arbitrarily large moduli of isomorphism classes of Einstein and non-Einstein nilradicals, see Section \ref{section 2} for a more precise statement.  The size of these moduli depends on $q$, but goes to infinity as $q$ does.

\section{Concatenating Structure Matrices}\label{section 2}

We begin with an interesting question.  Let $\N = \rpqc$ be a two-step nilalgebra of type $(p,q)$.  Often $\N$ is decomposed as $\N=\V \oplus [\N,\N]$ where $\V$ is the orthogonal compliment of $[\N,\N]$.  The subspace $\V$ is naturally (isometrically) identified with $\mathbb R^q$.  Suppose we consider $C=(C^1,\dots , C^p)$ with the property that each $C^i$ preserves a common subspace of $\V$; that is, $\V = \V_1 \oplus \V_2$ where $\V_1$, and hence $\V_2$, is preserved by every $C^i$.  Let $q_i = dim \ \V_i$.  Then $q=q_1+q_2$ and the algebras $\N_i = \V_i \oplus \Z$ are of type $(p,q_i)$.

\begin{question*} Consider $\mathfrak N=\mathcal V_1\oplus \mathcal V_2 \oplus [\N,\N]$ as above.  Is $\mathfrak N$ an Einstein nilalgebra if and only if both $\mathfrak N_1$ and $\mathfrak N_2$ are so?
\end{question*}

This is a natural question as it asks whether or not the nilsoliton condition can be determined from the `irreducible' components of $\V$; here irreducibility is in the sense of representations.  Even though the answer is negative, this will be our approach to constructing moduli of both Einstein and non-Einstein nilradicals.  For more examples and information see \cite[Chapter 8]{Jablo:Thesis}

Consider  $A=(A_1, \dots , A_{p})\in \mathfrak{so}(q_1)^p$ and $B=(B_1,\dots ,B_{p})\in \mathfrak{so}(q_2)^{p}$ which are structure matrices associated to $\N_1$ and $\N_2$, where $q_i = \dim \V_i$.  Then $\N$ corresponds to the structure matrix $C\in \mathfrak{so}(q)^p$ where $q=q_1+q_2$ and $$C_i=\begin{pmatrix}A_i \\ & B_i \end{pmatrix}$$
We call this process \textit{concatenation} and denote it by $C= A+_cB$.  As $A$ and $B$ have linearly independent components, the same is true for  $C$ and hence $C$ corresponds to a nilalgebra of type $(p,q)$.

At times we will abuse notation and concatenate $A\in \mathfrak{so}(q_1)^{p_1}$ and $B\in \mathfrak{so}(q_2)^{p_2}$ where $p_1 < p_2$.  This is an element of $\mathfrak{so}(q_1+q_2) ^{p_2}$ defined as
    $$ (A_1,\dots , A_{p_1}, \underbrace{0,\dots,0}_{p_2-p_1} ) +_c (B_1,\dots, B_{p_2} )$$

\begin{defin}\label{def: SLp minimal}  Let $G$ be a reductive group acting linearly on a vector space $V$.  A point $v\in V$ is called $G$-minimal if $m(v)=0$ where $m$ is the moment map of the $G$ action on $V$.
\end{defin}

\textit{Remark}. Consider $SL(q,\mathbb R)\times SL(p,\mathbb R)$ acting on $\mathfrak{so}(q)^p$.  The distinguished points (see Definition \ref{def: distinguished point}) of this action are precisely the nilsoliton metrics, see the previous section for details.
We first note that the generic two-step nilsolitons of type $(p,q)$ with $p< D-2 = \frac{1}{2}q(q-1)-2$ are all $SL(p,\mathbb R)$-minimal.  We point out for completeness that there do exist distinguished points which are not $SL(p,\mathbb R)$-minimal.  See \cite[Chapter 8]{Jablo:Thesis}.

\begin{lemma} Let $A$ be a $SL(q,\mathbb R)\times SL(p,\mathbb R)$-distinguished  point which is $SL(p,\mathbb R)$-minimal.  Then $A$ is $SL(q,\mathbb R)$ distinguished.
\end{lemma}

\begin{proof}  This actually only requires $A$ to be $SL(p,\mathbb R)$ distinguished.  Recall that the moment map for the $SL(q,\mathbb R)\times SL(p,\mathbb R)$ action is $m=m_1 + m_2$ where $m_1$ is the moment map for $SL(q,\mathbb R)$ and $m_2$ is the moment map for $SL(p,\mathbb R)$.

Recall $A$ being distinguished is equivalent to $m(A)\cdot A = a A$ for some $a\in \mathbb R$.  But if $m_2(A)\cdot A = a_2 A$, then $m_1(A)\cdot A = (a-a_2)A$.  That is, $A$ is $SL(q,\mathbb R)$ distinguished.
\end{proof}

\begin{prop}\label{prop: concat SLp-minimal nilsolitons}  Consider $A \in \mathfrak{so}(q_1)^p$, $B \in \mathfrak{so}(q_2)^p $, and let
$C \in \mathfrak{so}(q_1+q_2)^p$ be the concatenation of $A$ and $B$.  If $A$, $B$ are distinguished and $SL(p,\mathbb R)$-minimal then so is $C$, after rescaling $B$.
\end{prop}

This gives a natural way of constructing new soliton algebras from smaller pieces.

\begin{proof}  We first observe that $A$ being $SL(p,\mathbb R)$-minimal is equivalent to $|A_i|=|A_j|$ and $A_i \perp A_j$ for all $i\not = j$.  Thus, if $A$ and $B$ are $SL(p,\mathbb R)$-minimal then the concatenation $C$ automatically is so, since $\langle C_i, C_j \rangle = \langle A_i, A_j \rangle + \langle B_i, B_j \rangle $.

By the lemma above, since $A$ and $B$ are $SL(p,\mathbb R)$-minimal, we see that $m_1(A)\cdot A = \lambda_aA$ and $m_1(B)=\lambda_bB$.  By rescaling, we may assume that $\lambda_a=\lambda_b$.  Let $C=\begin{pmatrix} A \\ & B \end{pmatrix}$ be the concatenation of $A$ and $B$.  Then

$$ m_1(C) = -2\sum C_i^2 = \begin{pmatrix} -2\sum A_i^2\\& -2 \sum B_i^2\end{pmatrix}= \begin{pmatrix}m_1(A)\\& m_1(B)\end{pmatrix}$$
and since we rescaled our initial pair, we see that

$$m_1(C)\cdot C = m_1(C)C+Cm_1(C)^t=\begin{pmatrix}m_1(A)\cdot A\\& m_1(B)\cdot B\end{pmatrix}=\lambda C$$
Since $C$ is $SL(p,\mathbb R)$-minimal, we see that $m_2(C)\cdot C = 0$.  Thus, $m(C)\cdot C=\lambda C$.
\end{proof}

\begin{thm}\label{thm: dimension of moduli q_1+q_2}  Suppose we form a family of concatenations $C=A+_c B$ by letting $A\in \mathfrak{so}(q_1)^p$ and $B\in \mathfrak{so}(q_2)^p$ vary.  Then the dimension of the moduli of such $C$ is bounded below by the sum of the dimensions of the moduli of such $A$ and $B$.
\end{thm}

\begin{proof} We will apply Theorem \ref{thm: finiteness thm and detecting orbits via H}.  Let $W$ be the subspace $ (\mathfrak{so}(q_1) \oplus \mathfrak{so}(q_2))^p \subset  \mathfrak{so}(q)^p$ where $q=q_1+q_2$.  And let $H = GL(q_1,\mathbb R) \times GL(q_2,\mathbb R) \times GL(p,\mathbb R) \subset GL(q,\mathbb R)\times GL(p,\mathbb R)$.  Then $G,H,V,W$ satisfy the hypotheses of Theorem \ref{thm: finiteness thm and detecting orbits via H}.

Thus, the moduli of $GL(q,\mathbb R)\times GL(p,\mathbb R)$-orbits containing such $C$ has the same dimension as the moduli of $GL(q_1,\mathbb R)\times GL(q_2,\mathbb R)\times GL(p,\mathbb R)$-orbits containing such $C$.  Obviously if $C=A_1+_cB_1 ,D=A_2+_cB_2 \in W$ are in the same $GL(q_1,\mathbb R)\times GL(q_2,\mathbb R)\times GL(p,\mathbb R)$-orbit, then $A_1,A_2$ are in the same $GL(q_1,\mathbb R)\times GL(p,\mathbb R)$-orbit, and similarly for the $B_i$.  This proves the theorem.
\end{proof}

\textit{Remark}.  It is possible to concatenate matrices in the same $GL(q_2,\mathbb R)\times GL(p,\mathbb R)$ orbit and obtain two concatenations which are not in the same $GL(q,\mathbb R)\times GL(p,\mathbb R)$ orbit.

\subsection*{Moduli of Einstein nilradicals via concatenation}
Let $A\in \mathfrak{so}(q_1)^p$ and $B\in \mathfrak{so}(q_2)^p$ be structure matrices corresponding to generic nilsolitons.  Then $C=A+_cB$ is also a nilsoliton (after rescaling $B$) by Proposition \ref{prop: concat SLp-minimal nilsolitons}.  We summarize this below.

\begin{thm}\label{thm: moduli of einstein nilradicals} Let $q\in \mathbb N$ and $p < D-2=\frac{1}{2}q(q-1)-2$.  Consider $\vpq^0$, the open set of $\mathfrak{so}(q)^p$ whose points correspond to nilalgebras.  Denote by $\mathcal M_{pq}$ the dimension of the moduli of nilsoltions computed about a generic point (see Theorem \ref{thm: generic two-step are einstein}).  Let $q_1,q_2\in \mathbb N$ be such that $q=q_1+q_2$, then there exist moduli of non-generic nilsolitons of dimension $\mathcal M_{p\ q_1} + \mathcal M_{p\ q_2} - 1$.
\end{thm}

This theorem is a combination of Theorem \ref{thm: dimension of moduli q_1+q_2}, Proposition \ref{prop: concat SLp-minimal nilsolitons}, and the remark following Definition \ref{def: SLp minimal}.  We subtract an additional one due to the rescaling of $B$ above. At first glance this may not seem interesting.  However, we point out the following corollary to contrast the generic setting of type $(2,2k+1)$ where the moduli has dimension zero about generic points.

\begin{cor}\label{cor: einstein nilrads of type (2,2k+1)}  Consider the algebras of type $(2,2k+1)$.  There exist moduli of non-generic algebras of this type which are nilsoliton and the moduli has dimension $k-i-4$ for all $i\leq k-4$.
\end{cor}

\begin{proof}  The proof amounts to picking matrices to concatenate.  Let $A$ be the generic nilsoliton of type $(2,2i+1)$.  This soliton is $SL(p,\mathbb R)$-minimal, for a construction of this see \cite[Chapter 7]{Jablo:Thesis}.  Choose $B$ to be a generic nilsoliton of type $(2,2k-2i)$.  These solitons are also $SL(p,\mathbb R)$-minimal and  the moduli of such has dimension $k-i-3$;  see the table in Theorem \ref{thm: generic two-step are einstein}.  Now apply the theorem above.
\end{proof}

\textit{Remark.} This technique will always build moduli of non-generic  Einstein nilradicals for all types $(p,q)$.  However, in general, the size of these constructed moduli will be smaller than the moduli of generic algebras.

\section{Moduli of non-Einstein nilradicals via concatenation}\label{section 3}

Finding examples of Lie algebras which cannot possibly admit a certain inner product is a very delicate problem. The first examples of moduli of non-Einstein nilradicals were constructed by Cynthia Will \cite{Will:CurveOfNonEinsteinNilradicals}.  Y. Nikolayevsky has recently classified the Einstein and non-Einstein nilradicals of type $(2,q)$ and as a consequence one obtains moduli of Einstein and non-Einstein nilradicals of this type.  This classification is produced by applying the (classical) systematic study of pencils of pairs of skew-symmetric matrices \cite{Nikolayevsky:EinsteinSolvmanifoldsattchedtoTwostepNilradicals}.  To our knowledge, these are the only other examples of moduli of non-Einstein nilradicals aside from those constructed here and in \cite{Jablo:FinitenessTheorem-compatibleSubgroups}.

In \cite{Jablo:FinitenessTheorem-compatibleSubgroups} it was declared that moduli of non-Einstein nilradicals arising from direct sums are somewhat trivial.  This is because of the following theorem; for a proof see that work.  The following theorem has also been proven independently by \cite{Nikolayevsky:EinsteinSolvmanifoldsandPreEinsteinDerivation} using different techniques.

\begin{thm}\label{thm: Einstein and direct sums} Consider a nilalgebra $\N = \N_1 \oplus \N_2$ which is a sum of ideals $\N_i$.  Then $\N$ is an Einstein nilradical if and only if $\N_1,\N_2$ are Einstein nilradicals.
\end{thm}

From this theorem one can easily construct moduli of non-Einstein nilradicals by finding one such Lie algebra $\N_1$ and then considering any family of nilalgebras $\N_2$.  This is trivial in some sense and so we are interested in algebras which are indecomposable; that is, those that do not decompose as a direct sum of ideals.

The curve of non-Einstein nilradicals given in \cite{Will:CurveOfNonEinsteinNilradicals} does not arise as a direct sum of ideals and so is non-trivial.  Moreover, Will's examples cannot arise as concatenations either, as ours do, and hence demonstrate the delicate nature of the question of whether or not a given nilalgebra is an Einstein nilradical.  (To see that Will's examples do not arise as concatenations, one can show that the subalgebra generated by the structure matrices is all of $\mathfrak{so}(6)$, which acts irreducibly on $\mathbb R^6$.)

\subsection*{Defining a class of   algebras of type $(p,q)$ for $2\leq p \leq 6$}

Recall that an algebra of type $(1,q)$ is isomorphic to a direct sum of an algebra of Heisenberg type plus an abelian algebra, i.e., the Euclidean de Rham factor; hence an algebra of type $(1,q)$ is an Einstein nilradical.  As $p$ is bounded above by $D_q = \frac{1}{2}q(q-1)$, and the (only) algebra of type $(D_q,q)$ is an Einstein nilradical, we will search for non-Einstein nilradicals of type $(p,q)$ with $2\leq p \leq D_q-1$.

We construct examples of moduli of non-Einstein nilradicals of type $(j,2k)$ with $j=2,\dots, 6$.  One  could create similar examples of type $(j,2k+1)$ by concatenating the  nilsoliton of type $(2,3)$ to our examples.

Denote by $J$ the $2\times 2$ matrix $\begin{bmatrix} 0 & 1\\ -1 &0\end{bmatrix}$.  Define $A_1 \in \mathfrak{so}(2k)$ to be the concatenation $A_1 = J\underbrace{+_c\dots +_c}_k J$.  This is just a block diagonal matrix whose blocks are all $J$'s.  Define $B_1,B_2,\dots, B_6 \in \mathfrak{so}(4)$ as
    $$B_1= \left[ \begin{BMAT}{cc.cc}{cc.cc} 0 &1 &&\\ -1&0 &&\\ && 0&1\\ && -1&0\end{BMAT} \right], \
    B_2 = \left[ \begin{BMAT}{cc.cc}{cc.cc} && 0&\ 1 \ \\ && \ 1\ &0\\ 0&-1&&\\ -1&0&&\end{BMAT} \right], \
    B_3 = \left[ \begin{BMAT}{cc.cc}{cc.cc} && \ 1 \ &0 \\ &&0& \ 1\ \\ -1&0&&\\ 0&-1&&\end{BMAT} \right]
    $$

    $$B_4= \left[ \begin{BMAT}{cc.cc}{cc.cc} 0 &1 &&\\ -1&0 &&\\ && 0&-1\\ && 1&0\end{BMAT} \right], \
    B_5 = \left[ \begin{BMAT}{cc.cc}{cc.cc} && 0&\ 1 \ \\ && -1\ &0\\ 0&\ 1&&\\ -1&0&&\end{BMAT} \right],\
    B_6 = \left[ \begin{BMAT}{cc.cc}{cc.cc} && \ 1 \ &0 \\ &&0& -1\ \\ -1&0&&\\ 0&\ 1&&\end{BMAT} \right]
    $$
Observe that these $B_i$ are all mutually orthogonal and $B_i^2=-Id$.
The choice of order of $B_i$ is made so as to decrease the work in Section \ref{section: indecomposability}; this choice is not necessary, only convenient.

We are interested in algebras whose structure matrices are of the form
    $$ C = A_1 +_c (t_1B_1, B_2) +_c \dots +_c (t_{n-1}B_1,B_2) +_c (B_1,B_2,\dots,B_j)$$
where $t_i\in \mathbb R$.  These algebras are type $(j, 2k+4n)$, for $j=2,\dots,6$.  As stated before, we are also interested in concatenating the soliton of type $(2,3)$ (cf. remarks at the end of this section) to this $C$ so as to construct algebras of type $(j,2k+4n+3)$.

\begin{thm}\label{thm: non-Einstein for p=2-6} Consider $C\in \mathfrak{so}(2k+4n+d)^j$ defined above where $n\geq 0$ and  $d=0$ or $3$.  For $2k\geq 4n+d$, the nilalgebra corresponding to $C$ is an indecomposable, non-Einstein nilradical.  Moreover, by varying $t_i$, we have an $(n-1)$-dimensional family of (pairwise) non-isomorphic, indecomposable, non-Einstein nilradicals.
\end{thm}

The proof is broken into 3 pieces.  First it is shown that varying $t=(t_1,t_2,\dots, t_{n-1})\in \mathbb R^{n-1}$ produces moduli of algebras.  Then it is shown that these algebras are non-Einstein nilradicals.  The proof of these pieces is an application of Theorem \ref{thm: finiteness thm and detecting orbits via H}.  The proof that these algebras are indecomposable is contained in Section \ref{section: indecomposability}.

\subsection*{Moduli of algebras}
We begin by showing that  varying $t = (t_1,t_2,\dots,t_{n-1})$ produces non-isomorphic algebras.  Only the details for the case $C\in \mathfrak{so}(2k+4n)^j$ are presented as the case $C\in \mathfrak{so}(2k+4n+3)^j$ is the same, mutatis mutandis.
We  apply Theorem \ref{thm: finiteness thm and detecting orbits via H}  and consider the vector space
    $$W = \{ a_1A_1 +_c (b_1B_1, c_1B_2) +_c\dots +_c (b_{n-1}B_1,c_{n-1}B_2)+_c (d_1B_1,\dots,d_jB_j)  | \ a_1, b_i,c_i,d_i \in \mathbb R\}$$
contained in $\subset V:=\mathfrak{so}(2k+4n)^j$.
The vector space $W$ is the subspace of block diagonal matrices of the same type as $C$.  Let $G = GL(2k+4n)\times GL(2)$ and let $H = \mathbb R^+\cdot Id_{2k}\times \underbrace{\mathbb R^+ \cdot Id_4 \times \dots \times \mathbb R^+\cdot Id_4}_n \times \ diag(j\times j) 
$ be the subgroup preserving the same subspaces of $\mathbb R^{2k+4n}$ that $C$ preserves.  Observe that $W$ is $H$-stable.  We will show that $G$ is $H$-detected along $W$ and apply Theorem \ref{thm: finiteness thm and detecting orbits via H}.

Consider $w =(C_1,\dots, C_j)= a_1 A_1 +_c (b_1B_1, c_1B_2) +_c\dots +_c (b_{n-1}B_1,c_{n-1}B_2)+_c (d_1B_1,\dots,d_jB_j) \in W$.  We compute $m_1(w)$ and $m_2(w)$ where $m=(m_1,m_2)$ is the moment map of the $G$-action on $V$ (cf. remarks following Theorem \ref{thm: Eberlein - isom classes}).
    \begin{eqnarray} \label{eqn: moment map m1 + m2}
    m_1(w) &=& -2\sum_i C_i^2 =    \left[\begin{BMAT}{cccccc}{cccccc}
        2a_1^2 Id_{2k}& &&&&\\
          &2(b_1^2+c_1^2)Id_4 &&&&\\
         && \ddots &&&\\
         &&& 2(b_{n-1}^2+c_{n-1}^2)Id_4&&\\
         &&&& 2(d_1^2 +\dots + d_j^2) Id_4   &\\
         &&&  & &
         \addpath{(0,5,.)rudrdlu}
         \addpath{(3,3,.)rldruddur}
    \end{BMAT}\right] 
    \notag \\
    m_2(w)&=& \begin{bmatrix} <C_i,C_j>\end{bmatrix} =
    \begin{bmatrix} 2ka_1^2 +4(b_1^2 +\dots +b_n^2 + d_1^2)  \\ & 4(c_1^2+\dots +c_n^2 +d_2^2) \\ && 4d_3^2 \\ &&& \dots \\ &&&& 4d_j^2 \end{bmatrix} 
    \end{eqnarray}
\textit{Remark.} We observe that the above work shows that $G=GL(q,\mathbb R)\times GL(p,\mathbb R)$ is $H$-detected along $W$, see Definition \ref{def: H detected along W}.  By Theorem \ref{thm: finiteness thm and detecting orbits via H} we know that for each $w\in W$, $G\cdot w \cap W$ is a finite union of $H$-orbits and hence finding moduli of $H$-orbits is equivalent to finding moduli of $G$-orbits.

\begin{prop} Consider $C[t]=A_1 +_c (t_1B_1, B_2) +_c \dots +_c (t_{n-1}B_1,B_2) +_c (B_1,B_2,\dots,B_j)$ as above where $t=(t_1,t_2,\dots , t_{n-1})$.  For $t\not = s$, the tuples of matrices $C[t],C[s]$ lie on distinct $H$-orbits.  That is, we have an $(n-1)$-dimensional moduli of $H$-orbits and hence an $(n-1)$-dimensional moduli of $G$-orbits.
\end{prop}

\begin{proof} Suppose $C[t] = h\cdot C[s]$ for some $h\in H$.  Write $h=(a_1,b_1,\dots ,b_n, diag(c_1,\dots,c_j))$.  Then $h\cdot C[s] = c_1\ a_1^2A_1 +_c (c_1\ b_1^2s_1B_1, c_2\ b_1^2B_2) +_c \dots +_c (c_1 \ b_{n-1}^2 s_{n-1} B_1 , c_2 \ b_{n-1}^2 B_2) +_c (c_1 \ b_n^2B_1,\dots, c_j \ b_n^2B_j)$ and $C[t] = h\cdot C[s]$ yield the following equalities (looking at the first and second slots)
    $$ \left\{ \begin{array}{l}
        t_i=c_1b_i^2s_i \\
        1= c_2b_1^2= \dots = c_2b_{n}^2
         \end{array}  \right.$$
This in turn implies $b_i^2 = 1/c_2$ which then implies $t_i = \frac{c_1}{c_2} s_i$ for all $i$.  As $t_n=s_n=1$, we have $\frac{c_1}{c_2}=1$ and so $t_i=s_i$ for all $i$; that is, $t=s$.

The last remark concerning moduli of $G$-orbits follows from Theorem \ref{thm: finiteness thm and detecting orbits via H}.
\end{proof}

\subsection*{Non-Einstein nilradicals}
Next we show  for $t = (t_1,t_2,\dots ,t_{n-1})$, with $t_i\not = 0$, that $C[t]$ cannot be an Einstein nilradical.  By Theorem \ref{thm: Einstein nilradical iff distinguished orbit}, $C[t]$ is an Einstein nilradical if and only if $G\cdot C[t]$ is a distinguished orbit.  By Theorem \ref{thm: finiteness thm and detecting orbits via H}, with the above work, this is equivalent to $H\cdot C[t]$ being a distinguished orbit.

Consider again the vector subspace
    $$W = \{ a_1A_1 +_c (b_1B_1, c_1B_2) +_c\dots +_c (b_{n-1}B_1,c_{n-1}B_2)+_c (d_1B_1,\dots,d_jB_j)  | \ a_1, b_i,c_i,d_i \in \mathbb R\}$$
Observe that $H$ preserves the open set of $W$ where $a_1,b_i,c_i,d_i \not = 0$.  If we can show that there are no distinguished points in this open set, then such $C[t]$ cannot have distinguished $H$-orbits and hence are non-Einstein nilradicals.

\begin{prop}  Consider the open set of $W$ defined above and suppose $2k \geq 4n$.  None of these points is a distinguished point for the action of $H$ on $W$ and hence none of these points has a distinguished $H$-orbit.  Consequently, all of these points correspond to algebras which are non-Einstein nilradicals.
\end{prop}

\begin{proof}  We only prove that none of these points is distinguished for the $H$-action on $W$ as the other claims of the proposition are addressed above.  In the sequel we write $b_n=d_1$ and $c_n=d_2$ to help with presentation.

We use our computations of the moment map above (see Equation \ref{eqn: moment map m1 + m2}).  For a point $w\in W$ to be distinguished we would have to have $m(w)\cdot w = rw$ for some $r\in \mathbb R$.  Comparing the `coefficients' in $m(w)\cdot w$ we have the following which must be equal
       \begin{eqnarray}
                4a_1^2         &+& 2ka_1^2 + 4 \sum_{i=1}^n b_i^2 \label{eqn: coeff 1}\\
                4(b_I^2+c_I^2) &+& 2ka_1^2 + 4 \sum_{i=1}^n b_i^2 \ \ \ \ \ \ \mbox{ (for } I=1,\dots,n-1) \label{eqn: coeff 2}\\
                4(b_I^2+c_I^2) &+& 4 \sum_{i=1}^n c_i^2 \ \ \ \ \ \ \ \ \ \ \ \ \ \ \ \ \mbox{ (for } I=1,\dots,n-1) \label{eqn: coeff 3}\\
                4(b_n^2+c_n^2) + 4(d_3^2+\dots +d_j^2) &+& 2ka_1^2 + 4 \sum_{i=1}^n b_i^2 \label{eqn: coeff 4}\\
                4(b_n^2+c_n^2) + 4(d_3^2+\dots +d_j^2) &+& 4 \sum_{i=1}^n c_i^2
                \label{eqn: coeff 5}\\
                4(b_n^2+c_n^2) + 4(d_3^2+\dots +d_j^2) &+& 4d_I^2 \ \ \ \ \ \ \ \ \ \ \ \ \ \ \ \ \ \ \ \ \ \ \ \mbox{ (for } I=3,\dots, j) \label{eqn: coeff 6}
                \end{eqnarray}
with $b_n=d_1$ and $c_n=d_2$.

Comparing Equations \ref{eqn: coeff 1} and \ref{eqn: coeff 2} yields $a_1^2 = b_i^2+c_i^2$ for $i=1,\dots, n-1$.  Comparing Equations \ref{eqn: coeff 1} and \ref{eqn: coeff 4} we have $a_1^2 = b_n^2+c_n^2 + \sum d_i^2$ which implies $a_1^2 \geq b_n^2+c_n^2$.  Putting these together we then have
    $$ n\ a_1^2 \geq \sum_{i=1}^n b_i^2+c_i^2$$
Comparing Equations \ref{eqn: coeff 2} and \ref{eqn: coeff 3} we have $2k\ a_1^2 + 4\sum b_i^2 = \sum c_i^2$, thus $2k \ a_1^2 = 4\sum (c_i^2-b_i^2)$.  Putting this together with the above yields
    $$ 2k\ \sum b_i^2+c_i^2 \leq 2k\ n\ a_1^2 = 4n\ \sum c_i^2-b_i^2$$
rearranging terms we have
    $$ 0 < (2k+n) \sum b_i^2 \leq (4n-2k) \sum c_i^2$$
By hypothesis, $b_i \not = 0$.  Thus, the above inequality cannot be satisfied when $2k\geq 4n$ and the proposition is proved.
\end{proof}

\textit{Remark.} Observe that Equation \ref{eqn: coeff 6} was not used in the proof above.  This will be used in the next section to build more non-Einstein nilradicals with $p>6$.

As stated through out this section, the above work can be adapted trivially to produce non-Einstein nilradicals of type $(j,2k+4n+3)$ where $j=2,\dots, 6$, $n\geq 1$, and $2k \geq 4n+3$.  We summarize our knowledge up to this point in the following figure; recall that algebras of type $(p,q)$ satisfy $p\leq \frac{1}{2}q(q-1)$.

\begin{figure}[htp]\label{fig: types 26}
\centering
\includegraphics[height=1.75in]{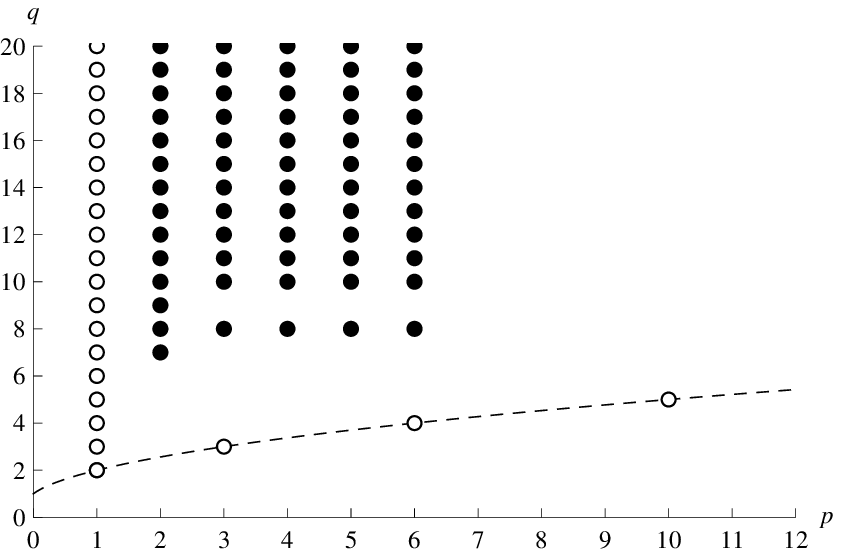}
\includegraphics[height=1.85in]{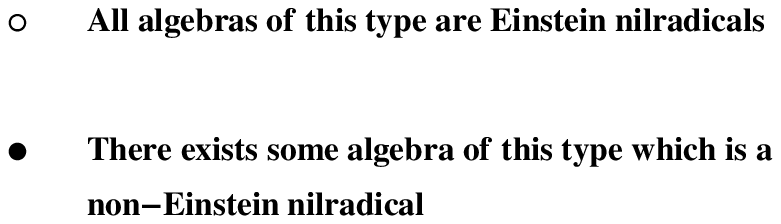}
\caption{Types $(p,q)$ which do or do not admit indecomposable non-Einstein nilradicals.}
\end{figure}

For the sake of completeness we will fill-in the missing types $(3,9),\dots, (6,9)$.
We only give the definition of the algebras of these types and leave to the reader the details of showing that they are non-Einstein
nilradicals.  The proof of this is similar to our earlier cases.  Our examples of these types are the following concatenations
    $$\left[ \J \right] +_c \left( \left[ \begin{BMAT}{c.c}{c.c}\J & \\&0 \end{BMAT}\right], \left[ \begin{BMAT}{c.c}{c.c}0&\\ &\J \end{BMAT} \right] \right) +_c (B_1,\dots,B_j)$$
where the $B_i\in \mathfrak{so}(6)$ are defined as before and $j = 3,\dots, 6$.  Observe that the middle term is the soliton of type $(2,3)$.

\section{Building new examples of non-Einstein nilradicals of type $(p,q)$}\label{section: examples of type (p,q)}

In this section we describe a procedure for building more non-Einstein nilradicals, even moduli of such, in types $(p,q)$ when $p>6$.  The examples built in this section will be shown to be indecomposable (see Section \ref{section: indecomposability}).  Presently our techniques for building new indecomposable nilalgebras do not work to build examples of such for all types $(p,q)$.  We summarize the results in Theorem \ref{thm: adjoin several extra tuples} and Corollary \ref{thm: summary of indecomp. of types (p,q)}.  The results are also displayed using Figure 2.  The examples presented below are built via concatenation from the preceding examples of type $(p,q)$ with $3\leq p\leq 6$.\\

\textit{Notation.}  Let $A\in \mathfrak{so}(q_1)^n$ and $B\in \mathfrak{so}(q_2)^m$.  Then $C=A+_a B$ will denote a tuple of matrices in $\mathfrak{so}(q_1+q_2)^{n+m-1}$ which is the following concatenation
    $$(A_1,\dots, A_n, \underbrace{0,\dots,0}_{m-1})+_c (\underbrace{0,\dots, 0}_{n-1},B_1,\dots, B_m)$$
If $A$ and $B$ correspond to nilalgebras then so does $C=A+_aB$.
By construction we have $C_n = A_n +_c B_1$; this overlap is chosen so that the corresponding nilalgebras will be indecomposable (see Section \ref{section: indecomposability}).  We call $A+_aB$ the \textit{adjoin} of $A$ and $B$.  

Similarly we can adjoin a set of matrices to another matrix, as follows.  Consider three structure matrices $A\in \mathfrak{so}(q_1)^{p_1}$, $B\in \mathfrak{so}(q_2)^{p_2}$, $C\in \mathfrak{so}(q_3)^{p_3}$.  The adjoin $A +_a \{B, C\} \in \mathfrak{so}(q_1+q_2+q_3)^{p_1+p_2+p_3-2} $ is the following concatenation
    $$(A_1,\dots, A_{p_1}, \underbrace{0, \dots, 0}_{p_2+p_3-2}) +_c ( \underbrace{0,\dots, 0}_{p_1-1}, B_1,\dots,B_{p_2},\underbrace{0\dots,0}_{p_3-1} )+_c ( \underbrace{0,\dots,0}_{p_1-1} ,C_1, \underbrace{0,\dots,0}_{p_2-1},C_2,\dots, C_{p_3} )  $$
This construction deliberately overlaps a matrix from each tuple in  the $p_1$-slot.  This is done so as to insure that the adjoin of indecomposable matrices is again indecomposable (see Section \ref{section: indecomposability}).  Adjoining more than two tuples to another matrix is done analogously.  Although these constructions might appear technical at first, they are engineered to minimize the amount of calculation in the sequel.\\

Consider the examples $C \in \mathfrak{so}(q,\mathbb R)^j$, with $j=3,\dots,6$, which were constructed in the previous section.  We build our new and bigger tuples by adjoining more matrices to such $C$.  
Let $D\in \mathfrak{so}(q)^p$ be a minimal point of $SL_q \mathbb R \times SL_p\mathbb R$ action on $\mathfrak{so}(q)^p$; this is equivalent to the geometric condition that the metric algebra $\mathbb R^{p+q}[D]$ be a Ricci soliton metric which is also geodesic flow invariant (see \cite{Eberlein:prescribedRicciTensor}).  Furthermore, we assume that $D_1$ satisfies $D_1^2= -Id$, hence $q$ is even.

Recall that $D$ is a minimal point of the $SL_q\mathbb R \times SL_p\mathbb R$ action if and only if
    $$ m_1(D)=-2\sum D_i^2 = r Id \ \ \ \mbox{ and } \ \ \  m_2(D)=[<D_i,D_j>]_{ij}=s Id$$
for some $r,s\in \mathbb R$ (cf. Section \ref{section: two-step nilpotent lie algebras and RpqC}).  As a consequence of the assumption that $D_1^2=-Id$, we see that $\tilde D = (\lambda  D_1, \mu D_2, \mu D_3, \dots, \mu D_q)$ satisfies
    $$ \ \ \ \ \ \ \ \ \ \ m_1(\tilde D) = a Id  \ \ \  \mbox{ and } \ \ \ m_2(\tilde D) =diag(b, c, \dots ,c )  $$
Letting the first entry of $D$ rescale by itself will be necessary in the work that follows.

\begin{thm}\label{thm: adjoin one extra tuple} Consider one of our non-Einstein nilradicals $C\in \mathfrak{so}(q_1)^j$, constructed in the previous section, with $q_1 \geq 8$, $j=3,\dots,6$.  Let $D\in \mathfrak{so}(q_2)^p$ be a minimal point for the $SL_q\mathbb R\times SL_p\mathbb R$ action such that $D_1^2=-Id$ and $p$ satisfies $\frac{1}{2}(q_2-2)(q_2-3)+2\leq p \leq \frac{1}{2}q_2(q_2-1)$.  The algebra corresponding to $C+_aD$ is an indecomposable, non-Einstein nilradical of type $(j+p_2-1, q_1+q_2)$.
\end{thm}

\textit{Remarks.} 1) The restrictions placed on the size of $D$ are chosen so as to guarantee that $D$, and hence $C+_aD$, is an indecomposable algebra (see Section \ref{section: indecomposability}).

2) By varying $C$ one easily obtains non-isomorphic algebras $C+_aD$.  We leave these details to the reader as they are similar to earlier work.

\begin{proof}  The proof of indecomposablity is contained in Section \ref{section: indecomposability}.  We prove here that $C+_aD$ is non-Einstein.
As in the previous section we will exploit Theorem \ref{thm: finiteness thm and detecting orbits via H} by finding a subspace $W$ and a group $H$ along which $G= GL(q_1+q_2)\times GL(j+p-1)$ is $H$-detected.  We present the details when $C\in \mathfrak{so}(2k+4n)^j$ as the other cases are identical, mutatis mutandis.\\

Consider a vector space
    $$W = \{ a_1A_1 +_c (b_1B_1, c_1B_2) +_c\dots +_c (b_{n-1}B_1,c_{n-1}B_2)+_c (d_1B_1,\dots,d_jB_j)+_a (\lambda D_1, \mu D_2,\dots,\mu D_p)\}$$
where $ a_1, b_i,c_i,d_i,\lambda,\mu \in \mathbb R$.  The only difference between this subspace and the $W$ used in the previous section is that we have adjoined a small piece which is two dimensional.  Let $H$ be the subgroup of $G=GL(q_1+q_2)\times GL(j+p-1)$ which rescales the various components of $W$, that is,
    $$(\mathbb R \cdot Id_{2k} \times \mathbb R \cdot Id_4 \times \dots \times \mathbb R\cdot Id_4 \times \mathbb R \cdot Id_{q_2})\times (diag(j\times j)\times \mathbb R\cdot Id_{p-1})$$
This group preserves $W$ and, moreover, we will show that $G$ is $H$-detected along $W$.

Consider  $w =(C_1,C_2,\dots, C_p)= a_1A_1 +_c (b_1B_1, c_1B_2) +_c\dots +_c (b_{n-1}B_1,c_{n-1}B_2)+_c (d_1B_1,\dots,d_jB_j)+_a (\lambda D_1, \mu D_2,\dots,\mu D_p)$.  We compute $m_1(w)$ and $m_2(w)$ where $m=(m_1,m_2)$ is the moment map of the $G$-action on $V$ (cf. remarks following Theorem \ref{thm: Eberlein - isom classes}).
    \begin{eqnarray*}
    m_1(w) &=& -2\sum_i C_i^2 =    \left[\begin{BMAT}{cccccc}{cccccc}
        2a_1^2 Id_{2k}& &&&&\\
          &2(b_1^2+c_1^2)Id_4 &&&&\\
         && \ddots &&&\\
         &&& 2(b_n^2+c_n^2)Id_4&&\\
         &&&& 2(d_1^2 +\dots + d_j^2) Id_4   &\\
         &&&&& 2r\ Id_{q_2}
         \addpath{(0,5,.)rudrdlu}
         \addpath{(3,3,.)rldruddurdlrrld}
    \end{BMAT}\right] 
    \notag \\
    m_2(w)&=& \begin{bmatrix} <C_i,C_j>\end{bmatrix} =
    \begin{bmatrix} 2ka_1^2 +4(\sum_{i=1}^n b_i^2)  \\ & 4(\sum_{i=1}^n c_i^2 ) \\ && 4d_3^2 \\ &&& \ddots \\ &&&& 4d_j^2+\lambda^2|D_1^2| \\ &&&&& \mu^2|D_1^2| \\ &&&&&& \ddots \\ &&&&&&& \mu^2|D_1^2| \end{bmatrix} 
    \end{eqnarray*}
As before we are simplifying presentation by writing $b_n=d_1$ and $c_n=d_2$. We are also using the fact that $D$ being a minimal point implies $|D_1|=|D_2|=\dots=|D_p|$.
Moreover, we have applied the observations concerning $\tilde D = (\lambda D_1, \mu D_2,\dots ,\mu D_p)$ which precede this theorem.

We observe that these computations show that $G=GL(q_1+q_2,\mathbb R)\times GL(j+p-1,\mathbb R)$ is $H$-detected along $W$, see Definition \ref{def: H detected along W}.  By Theorem \ref{thm: finiteness thm and detecting orbits via H} we know that for each $w\in W$, $G\cdot w \cap W$ is a finite union of $H$-orbits and hence finding moduli of $H$-orbits is equivalent to finding moduli of $G$-orbits.  Moreover, the $G$ orbit of $w$ is distinguished if and only if the $H$ orbit is distinguished.  We complete our proof by showing that the $H$ orbit of $w$ is not distinguished.  This is done using Equations \ref{eqn: coeff 1} -- \ref{eqn: coeff 6}.

Consider the equation $m(w)\cdot w = \lambda w$, the condition to be a distinguished point.  As in the previous section, we consider the `coefficients' of this equation.  If $w$ were a distinguished point, we would obtain the following set of equalities (this is not the complete list, only the portion needed for our proof).  This set differs from Equations \ref{eqn: coeff 1} -- \ref{eqn: coeff 6} only in the last two equations.

\begin{eqnarray}
                4a_1^2         &+& 2ka_1^2 + 4 \sum_{i=1}^n b_i^2 \\
                4(b_I^2+c_I^2) &+& 2ka_1^2 + 4 \sum_{i=1}^n b_i^2 \ \ \ \ \ \ \mbox{ (for } I=1,\dots,n-1) \\
                4(b_I^2+c_I^2) &+& 4 \sum_{i=1}^n c_i^2 \ \ \ \ \ \ \ \ \ \ \ \ \ \ \ \ \mbox{ (for } I=1,\dots,n-1) \\
                4(b_n^2+c_n^2) + 4(d_3^2+\dots +d_j^2) &+& 2ka_1^2 + 4 \sum_{i=1}^n b_i^2 \\
                4(b_n^2+c_n^2) + 4(d_3^2+\dots +d_j^2) &+& 4 \sum_{i=1}^n c_i^2\\
                4(b_n^2+c_n^2) + 4(d_3^2+\dots +d_j^2) &+& d_I^2 \ \ \ \ \ \ \ \ \ \ \ \ \ \ \ \ \ \ \ \ \ \ \ \mbox{ (for } I=3,\dots, j-1) \\
                4(b_n^2+c_n^2) + 4(d_3^2+\dots +d_j^2) &+& d_j^2+\lambda^2
                \end{eqnarray}
with $b_n=d_1$ and $c_n=d_2$.

Observe that the first five equations are only conditions on $C$ and these are precisely Equations \ref{eqn: coeff 1}--\ref{eqn: coeff 5}.  Moreover, in the proof of Theorem \ref{thm: non-Einstein for p=2-6} only these five equations were used to show that $C$ was not an Einstein nilradical.  Hence, the $H$ orbit of $C+_aD$ is not distinguished and so $C+_aD$ is a non-Einstein nilradical.
\end{proof}

\subsection*{More general procedures}

We observe that the bulk of the work in proving Theorem \ref{thm: adjoin one extra tuple} is reducing to the problem to the proof of Theorem \ref{thm: non-Einstein for p=2-6}.  Closer inspection of the proof above reveals that much more can be accomplished.

Let $\{D^1,\dots,D^k\}$ be a collection of tuples $D^i\in \mathfrak{so}(q_i)^{p_i}$ which satisfy the hypothesis on $D$ in Theorem \ref{thm: adjoin one extra tuple}; that is, each $D^i$ is a minimal point for the $SL_{q_i}\times SL_{p_i}$ action, $D^i_1$ squares to a multiple of the identity, and $p_i\in [ \frac{1}{2}(q_i-2)(q_i-3)+2, \frac{1}{2}q_i(q_i-1) ]$.  Then the following is true.

\begin{thm}\label{thm: adjoin several extra tuples}  Consider one of the non-Einstein nilradicals $C\in \mathfrak{so}(q)^j$ defined in the previous section with $q\geq 8$ and $j\in [3,6]$.  The adjoin $C+_a \{D^1,\dots, D^k \}$ is an indecomposable, non-Einstein nilradical of type $(j+p_1+\dots +p_k-k,q+q_1+\dots+q_k)$.
\end{thm}

The proof of this theorem is a slight modification of the proof of Theorem \ref{thm: adjoin one extra tuple}.  The details are left to the reader.\\

\textit{Remark}.  Such $D$ always exist for $q$ even and $p=\frac{1}{2}q(q-1)$ and $p=\frac{1}{2}q(q-1)-1$.  To see this start with a matrix $D_1$ whose square is $-Id$.  Then extend $D_1$  to an orthogonal basis of $\mathfrak{so}(q)$ whose elements all have the same length.  This gives a tuple with $p=\frac{1}{2}q(q-1)$ which satisfies the desired hypotheses of Theorem \ref{thm: non-Einstein for p=2-6}.  To extend this construction to smaller $p$, one would just choose elements $D_2, D_3,\dots$ of the basis with the property that they also square to $-Id$.  By removing these elements from the list of basis elements one obtains tuples with the desired properties.\\

Applying the theorem above with each tuple $D^i$ as a subset of $\{B_1,\dots,B_6\}$ defined in Section \ref{section 3} we obtain

\begin{cor} \label{thm: summary of indecomp. of types (p,q)}  For $(p,q)$ satisfying $8\leq q$ and $2\leq p \leq \frac{5}{4}q-8$, there exist indecomposable non-Einstein nilradicals of type $(p,q)$.  Moreover, most of these types admit moduli of such algebras.  The dimension of said moduli is bounded below by $\frac{1}{8}[q-\frac{4}{5}(p+8)-7]$.
\end{cor}

The corollary is clearly very crude as can be seen in the figures (below) which summarize our current knowledge.  The lower bound on the dimension is easily derived from the procedure by which these algebras are constructed.

\begin{figure}[htp]\label{fig: general types}
\centering
\includegraphics[height=1.75in]{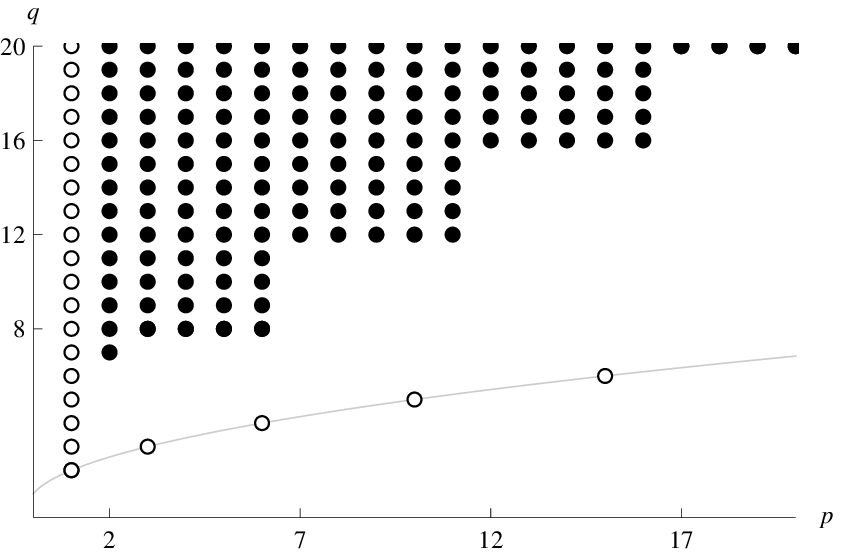}
\hfill
\includegraphics[height=1.75in]{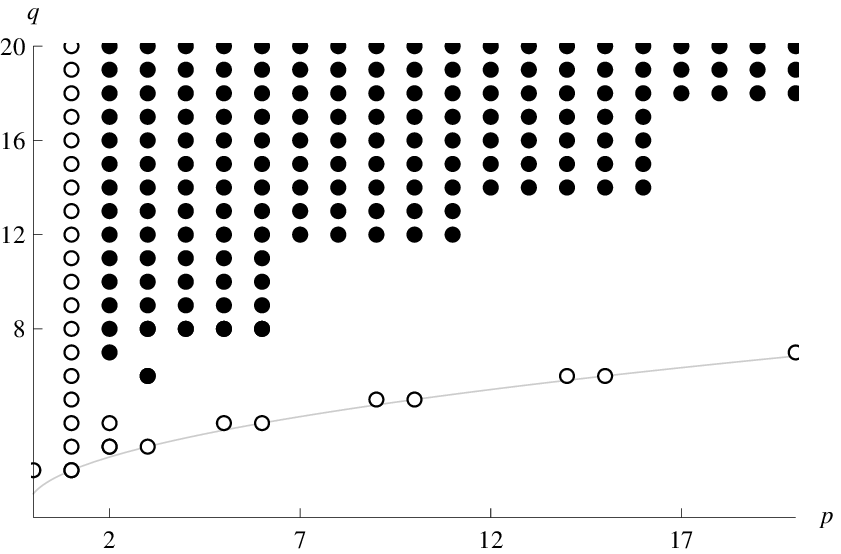}
\caption{Types $(p,q)$ which do or do not admit indecomposable non-Einstein nilradicals.}
\end{figure}

\textit{Remarks}. (1)  As before (cf. Figure 1) a circle represents a type $(p,q)$ for which all algebras are Einstein nilradicals while a disk represents a type $(p,q)$ for which there exists at least one non-Einstein nilradical.

(2)  The left image in Figure 2 represents the work that has been presented thus far in this paper.  The right image in Figure 2 represents some results from other papers and more examples that can be squeezed out by our techniques.  In \cite{Will:CurveOfNonEinsteinNilradicals} one can find the fact that all nilpotent Lie algebras of dimension $\leq 6$ are Einstein nilradicals and the construction of a non-Einstein nilradical of type $(3,6)$.  In \cite{Nikolayevsky:EinsteinSolvmanifoldsattchedtoTwostepNilradicals} one may find the fact that all algebras of type $(D-1,q)$ are Einstein nilradicals.\\

\section{Indecomposability}\label{section: indecomposability}

In this section we address the issue of indecomposability; i.e., showing that all the examples constructed above cannot decompose as a direct sum of ideals.  One way to construct a direct sum of ideals, using structure matrices, is as follows.

Let $\N_1, \N_2$ be   two-step nilpotent Lie algebras with adapted bases $\{ X_1, \dots , X_q, Z_1,\dots , Z_p \}$ and $\{ Y_1, \dots , Y_r ,$ $ W_1 , \dots , W_s \} $, respectively.  From these bases we can construct structure matrices $A\in \mathfrak{so}(q)^p$ and $B\in \mathfrak{so}(r)^s$.  Now define $\tilde A = (A, 0,\dots ,0) \in \mathfrak{so}(q)^{p+s}$ and $\tilde B = (0,\dots, 0, B)\in \mathfrak{so}(r)^{p+s}$.  Then the concatenation $\tilde A +_c \tilde B \in \mathfrak{so}(q+r)^{p+s}$ is the structure matrix of $\N_1 \oplus \N_2$ corresponding to the adapted basis $\{ X_i, Y_i, Z_j,W_j\}$.

Essentially this is the only way to concatenate  matrices to achieve a direct sum of ideals, as we show below.  In the work that follows, we assume that our nilalgebras have no Euclidean de Rham factor; i.e., the commutator is the center.  All of the examples used in this work satisfy this hypothesis.

\begin{lemma}\label{lemma: V=V_1+V_2 as kernels}  Let $\N=\V + \Z$ be a nilalgebra of type $(p,q)$.  Assume that $\N$ has no Euclidean deRham factor, that is, $[\N,\N]=\Z$ is the center.  Let $C = (C_1,\dots,C_p)\in \mathfrak{so}(q)^p$ be a set of structure matrices of $\N$  corresponding to some adapted basis.

Then $\N$ is decomposable as a direct sum of ideals if and only if there exists a basis $\{A_1\dots, A_l\} \cup \{ B_1 , \dots, B_k\}$ of $span<C_1,\dots, C_p>\subset \mathfrak{so}(q)$ such that $\V = \V_1  \oplus \V_2 $ where $\V_2  = \cap_{i=1}^l Ker \ A_i$ and $\V_1  = \cap_{j=1}^k Ker \ B_j$.
\end{lemma}

\textit{Remarks.} (1) It is not true that $\V_1$ must be stable under $\{A_i\}$.  This is because the direct sum $\V_1\oplus\V_2$ is not necessarily orthogonal.  One could choose an inner product so that this direct sum is orthogonal.  However,  fixing an inner product, a priori, it is easy to construct examples of direct sums of ideals which cannot decompose as an orthogonal direct sum of ideals.

(2)  Here we have assumed that $\Z=[\N,\N]$ as otherwise the algebra would decompose as a sum of ideals in a trivial way.  Moreover, having no Euclidean de Rham factor is precisely the condition to achieve the equalities $\V_2  = \cap_{i=1}^l Ker \ A_i$ and $\V_1  = \cap_{j=1}^k Ker \ B_j$; when the Euclidean de Rham factor is present one only has containment in one direction.

\begin{proof}  Let $C =(C_1,\dots ,C_p)$ be a set of structure matrices with respect to some adapted basis for $\N$.

Suppose that $\N$ is decomposable as a sum of ideals $\N_1 \oplus \N_2$.    Let $\{X_1,\dots, X_l,Z_1,\dots,Z_K\}$ be an adapted basis for $\N_1$ and let $\{X_{l+1},\dots,X_q,Z_{K+1},\dots,Z_p\}$ be an adapted basis for $\N_2$.  Taking the union of these bases we have an adapted basis for $\N$, see Section \ref{section: two-step nilpotent lie algebras and RpqC}.  Let $\{\tilde A_1,\dots, \tilde A_l,\tilde B_1,\dots,\tilde B_{p-l}\}$ denote the set of structure matrices for $\N$ with respect to this choice of adapted basis.  We know that there exists $(g,h)\in GL(q,\mathbb R)\times GL(p,\mathbb R)$ such that $(g,h)\cdot (\tilde A_1,\dots) = C$; that is, $\{g\cdot \tilde A_i\}\cup\{g\cdot \tilde B_j\}$ forms a basis of the span $<C_1,\cdots,C_p>\subset \mathfrak{so}(q)^p$.  Define $A_i=g\cdot \tilde A_i = g\tilde A_ig^t$ and $B_i=g\cdot \tilde B_i = g\tilde B_ig^t$.

Define $\tilde \V_1 = span<e_1,\dots, e_l>$, $\tilde \V_2=span<e_{l+1},\dots,e_q> \subset \mathbb R^q$ and $\V_1 = (g^t)^{-1}\tilde \V_1$, $\V_2 = (g^t)^{-1}\tilde \V_2$.  Observe that $\tilde \V_1 = \cap Ker\ \tilde B_i$, $\tilde \V_2=\cap Ker\ \tilde A_i$ and, hence, $\V_1=\cap Ker\ B_i$, $\V_2=\cap Ker\ A_i$.  This proves one direction of the theorem.

Now suppose that we have a two-step nilalgebra $\N=\V\oplus \Z$, a set of structure matrices $C=( A_1,\dots,  A_K,$ $ B_1,\dots, B_{p-K})$ (with respect to some adapted basis), and $\V_1,\V_2$ such that $\V=\V_1\oplus \V_2$ and $\V_1=\cap Ker\ B_i$, $\V_2=\cap Ker\ A_j$.  Choose $g\in GL(q,\mathbb R)$ such that $\tilde \V_1 = g\cdot \V_1 = span<e_1,\dots,e_l>$ and $\tilde \V_2 = g\cdot \V_2 = span<e_{l+1},\dots,e_q>$.  Define $\tilde C = (g^t)^{-1}\cdot C$.

Consider the metric two-step nilalgebra $\mathbb R^{p+q}[\tilde C]$ which is isomorphic to $\N$.  Recall that $\mathbb R^{p+q}[\tilde C]$ has the inner product so that $\{e_1,\dots, e_q,e_{q+1},\dots,e_p\}$ is orthonormal and the bracket relations are defined by $<[e_i,e_j],e_{q+k}>=(\tilde C_k)_{ij}$.  Now observe that $\tilde \V_2 = \cap Ker\ \tilde A_i$ being stable under each $\tilde A_i$ implies that $\tilde \V_1 = \tilde \V_2^\perp$ is stable under each $\tilde A_i$.  Similarly, $\tilde \V_2$ is stable under each $\tilde B_j$.  As the Lie bracket is described by $<[e_i,e_j],e_{q+k}>=(\tilde C_k)_{ij}$ and $(\tilde C_k)_{ji}=<\tilde C_k\ e_i,e_j>$, we see that $\mathbb R^{p+q}[\tilde C] = \N_1 \oplus \N_2$ where $\N_1 = \tilde \V_1\oplus span<e_{q+1},\dots, e_{q+K}>$ and $\N_2 = \tilde \V_2\oplus span<e_{q+K+1},\dots,e_{q+p}>$.
\end{proof}

\begin{prop} Let $A\in \mathfrak{so}(q_1)^{p_1}$ and $B\in \mathfrak{so}(q_2)^{p_2}$ be indecomposable algebras with $p_1<p_2$.  Then the concatenation $A+_cB \in \mathfrak{so}(q_1+q_2)^{p_2}$ is also indecomposable.
\end{prop}

The proof of this proposition is so similar to that of the next proposition, that we omit the details.  Moreover, once this result is known for the concatenation of two tuples, it is true via induction for a concatenation of an arbitrary number of tuples.

\begin{prop} Let $A^l  \in\mathfrak{so}(q_l)^{p_l}$, for $l=1,\dots, n$, be a collection of structure matrices corresponding to indecomposable algebras.  The algebra corresponding to the adjoin $C=A^1+_a\dots+_aA^n$ is also indecomposable.  \end{prop}

\textit{Remark.}  Observe that this implies $A^1 +_a \{ A^2,\dots, A^n\}$ is also indecomposable.

\begin{proof}  Using induction, we can reduce to the case $n=2$, that is, $C=A^1+_aA^2$.

As the elements $C_i$ of $C$ consist of block matrices, we have an orthogonal decomposition $\mathbb R^q = U_1 \oplus U_2$, with $U_j\simeq \mathbb R^{q_j}$, which is preserved by all $C_i$.  Here $A^l$ preserves $U_l$ and vanishes on $U_i$, $i\not =l$, and from this we immediately see that any matrix in the span of $<C_1,\dots, C_p>$ also preserves the decomposition $U_1\oplus U_2$.

Now suppose that $C$ is decomposable.  Lemma \ref{lemma: V=V_1+V_2 as kernels} states that there exists a basis $\{Z_i\} \cup \{ W_j\} \subset <C_1,\dots, C_p>$ and $\V_1,\V_2 \subset \mathbb R^q$ satisfying $\mathbb R^q = \V_1 \oplus \V_2$, $\V_1=\cap Ker\ W_j$, and $\V_2=\cap Ker\ Z_i$.  Observe that $Ker\ Z_i = \oplus_{j=1}^2 (U_j\cap Ker\ Z_i)$.  As $\V_2 = \cap_i Ker\ Z_i$ and each $U_j$ is stable under each $Z_i$, we immediately obtain
    $$\V_2 = \oplus_{j=1}^2(U_j \cap (\cap_i Ker\ Z_i))=\oplus_{j=1}^2(U_j\cap \V_2)$$
Similarly one can use $W_i$ to obtain
    $$\V_1 = \oplus_{j=1}^2(U_j\cap \V_1)$$
These direct sums are orthogonal.  Together these show $U_j = (U_j\cap \V_1)\oplus (U_j \cap \V_2)$ for $j=1,2$.  A priori, this direct sum is not necessarily orthogonal.

Next one can project each $Z_i$, $W_k$ onto the $A^l$-block; this is just restriction to $U_l$. Upon doing this we can extract from this set  a basis of $span<A^l_1,\dots,A^l_{p_l}>$.  We have three cases:  1) each $Z_i$ projects to zero,  2) each $W_k$ projects to zero, 3) there exists at least one $Z_i$ and one $W_k$ with non-zero projection.

Case 1.  As $\{Z_i\} \cup \{W_k\}$ form a basis of the span $<C_1,\dots, C_p>$, their projections span the  set  $<A^l_1,\dots, A^l_{p_l}>$.  Hence, the set $<A^l_1,\dots,A^l_{p_l}>$ is spanned by $\{  \restrictto{W_k}{U_l} \}$ and we have $\V_1 \cap U_l = \cap Ker \ \restrictto{W_k}{U_l} = \cap Ker\ A^l_k$.    From this we see that $U_l\cap \V_1 = \{0\}$ as $A^l$ being an indecomposable algebra there is no Euclidean de Rham factor (that is, no common kernel for the structure matrices).  Finally, since $U_l = (U_l\cap \V_1)\oplus (U_l\cap \V_2)$, we must have $U_l\cap \V_2 = U_l$.

Case 2.  Similar analysis shows that this case implies $U_l\cap \V_1 = U_l$.

Case 3.  Denote the extracted basis of the projection by $\{Z_i'\}\cup\{W_k'\}$.  Observe that $U_l\cap \V_1 = \cap Ker\ \restrictto{W_j}{U_l}
\subset \cap Ker\ W_k'$ and $U_l\cap \V_2= \cap Ker \ \restrictto{Z_j}{U_l} \subset \cap Ker\ Z_i'$.  If either of these set inclusions is not an equality, then $U_l = (U_l\cap \V_1)\oplus (U_l\cap \V_2)$ implies $( \cap Ker\ W_j') \cap (Ker\ Z_i') \not = \{0\}$ by comparing codimensions.  This says that there exists Euclidean de Rham factor which violates our hypothesis of $A^l$ being an indecomposable algebra.  Thus $U_l\cap \V_1 = \cap Ker\ W_j'$ and $U_l\cap \V_2 = \cap Ker\ Z_i'$.  But again using the hypothesis of $A^l$ being an indecomposable algebra, we achieve either $U_l\cap \V_1 =\{0\}$ or $U_l \cap \V_2 = \{0\}$ (cf. Lemma \ref{lemma: V=V_1+V_2 as kernels}).  Thus, either $U_l\cap \V_1=U_l$ or $U_l\cap \V_2=U_l$.

Finally, we have shown, regardless of which case occurs, that $ U_l\cap \V_1=U_l$   or $U_l\cap \V_2=U_l$  for $l=1,2$.  This in turn implies
    $$ \V_1 = U_i  \mbox{ and } \V_2=U_j$$
where $\{i,j\} = \{1,2\}$.

Next we recall the process of building $C=A^1+_aA^2$.  Here $C\in \mathfrak{so}(q)^p$ where $q=q_1+q_2$ and $p=p_1+p_2-1$.  We describe the elements $C_k$ as block matrices which preserve the decomposition $\mathbb R^q = U_1 \oplus U_2$.  For $1\leq k \leq p_1-1$, $C_k = \begin{pmatrix} A^1_k \\ & 0\end{pmatrix}$,  $C_{p_1} = \begin{pmatrix} A^1_{p_1} \\ & A^2_1 \end{pmatrix}$, and for $p_1+1\leq k \leq p$, $C_k =\begin{pmatrix} 0\\ & A^2_{k-p_1+1}\end{pmatrix}$.

To finish the proof of the theorem, we assume without loss of generality that $\V_1 = U_1$ and $\V_2=U_2$; that is, our collection of $\{Z_i\}$ vanishes on $U_2$ while our collection of $\{W_k\}$ vanishes on $U_1$.  Choose any $Z\in \{Z_i\}$ and write $Z = \sum a_k C_k$.  This matrix vanishes on $U_2$ and so we have
    $$0 = Z|_{U_2} = \sum_{k=1}^p a_k C_k|_{U_2} = \sum_{k=p_1}^p a_k C_k|_{U_2}=\sum_{k= 1}^{p_2} (a_{k+p_1-1}) A^2_k$$
However, the $\{A^2_k\}$ are linearly independent and hence $a_k=0$ for $k\geq p_1$; that is, the span $<Z_i>$ is contained in the span $<C_1,\dots, C_{p_1-1}>$.  Similarly, using the $\{W_j\}$ one obtains that the span $<W_j>$ is contained in the span $<C_{p_1+1}, \dots, C_p>$. Thus, $C_{p_1}$ is not in the span $<Z_i, W_j>$ which violates the hypothesis that $\{Z_i\}\cup \{W_j\}$ is a basis of the span $<C_1,\dots,C_p>$.

This proves that $C=A^1+_aA^2$ is indecomposable if both $A^l$ are indecomposable.  Using induction, the theorem is true for $n\geq 2$.

\end{proof}

\begin{cor}  The algebras constructed in this work are indecomposable.\end{cor}

\begin{proof}  The algebras constructed in this work are built by concatenating and  adjoining the structure matrices of indecomposable algebras.  If the smaller algebras are indecomposable, then the propositions above say that the concatenation or adjoin will also be indecomposable.  The following lemmas show that the input algebras are all indecomposable.
\end{proof}

\begin{lemma}  Consider $B=(B_1,B_2)\in\mathfrak{so}(3)^2$.  The algebra corresponding to $B$ is indecomposable.
\end{lemma}

\begin{proof}  We apply Lemma \ref{lemma: V=V_1+V_2 as kernels}.  Suppose there exists a basis $\{Z_1, W_1\}$ of the span $<B_1,B_2>$ such that $\mathbb R^3 = \V_1\oplus \V_2$ where $\V_1= \ Ker Z_1$ and $\V_2 = Ker \ W_1$.  Then $\dim \V_i = 1$ or $2$.  However, since $\mathbb R^3$ is odd dimensional, the $\V_i$ must be odd dimensional as they are kernels.  Thus no such $Z_1,W_1$ exist.
\end{proof}

\begin{lemma}\label{lemma: B preserves no proper subspaces of R4} Consider $B=(B_1,B_2)\in \mathfrak{so}(4)^2$ where $B_1 = \begin{bmatrix} J & 0\\ 0 & J\end{bmatrix}$ and $B_2 =\begin{bmatrix} 0& K\\ -K&0\end{bmatrix}$ are block matrices with $J=\begin{bmatrix}0&1\\-1&0\end{bmatrix}$ and $K=\begin{bmatrix}0& 1\\1&0\end{bmatrix}$.  Any non-trivial linear combination is non-singular, hence $B$ corresponds to an indecomposable algebra.
\end{lemma}

The proof is a simple calculation, which we omit, and the application of Lemma \ref{lemma: V=V_1+V_2 as kernels}.

\begin{lemma}\label{lemma: indecom for large p} Consider $C\in \mathfrak{so}(q)^p$ with $q$ even.  If $\frac{1}{2}(q-2)(q-3)+2\leq p \leq \frac{1}{2}q(q-1)$, then $C$ is an indecomposable algebra.
\end{lemma}

\textit{Remark}.  For the case $q=6$ this produces the bounds $3\leq p\leq 6$.

\begin{proof}  Recall that the kernel of any element of $\mathfrak{so}(q)$ will be even dimensional as $q$ is even.

Again we apply Lemma \ref{lemma: V=V_1+V_2 as kernels}.  Suppose that $C$ corresponds to a decomposable algebra.  Let $\{A_1,\dots,A_L\}$, $\{B_1,\dots,B_K\}$, $\V_1 = \cap Ker\ A_i$, and $\V_2 = \cap Ker\ B_j$ be as in that lemma.  For such a sum $\V_1 \oplus \V_2$ to be non-trivial, one would have $\dim \V_1, \dim \V_2 \geq 1$ and $\dim \V_1 + \dim \V_2=q$.  This in turn implies $\dim \V_1 , \dim \V_2 \geq 2$.  To see this, suppose $\dim \V_1=1$.  Then $\dim (\cap Ker\ B_j)=q-1$ and hence $\dim Ker\ B_j \geq q-1$ for all $j$.  However, $\dim Ker\ B_j$ must be even, and thus we would have $\dim Ker\ B_j = q$, which is a contradiction.

Now that we have established the inequalities $\dim \V_1, \dim \V_2 \geq 2$ we can proceed.  Let $M=\dim \V_1$, then we have $q-M=\dim \V_2$.  As $\V_2\subset Ker\ A_i$, $A_i$ preserves $\V_2^\perp$ and hence $A_i\in \mathfrak{so}(\V_2^\perp)$.  And so we have $L = card\{A_1,\dots,A_L\} \leq \dim \mathfrak{so}(\V_2^\perp)=\dim \mathfrak{so}(M)=\frac{1}{2}M(M-1)$.  Similarly we have $K=card\{B_1,\dots,B_K\} \leq \frac{1}{2}(q-M)(q-M-1)$.

As $\{A_1,\dots,A_L\}\cup \{B_1,\dots,B_K\}$ form a basis of the span $<C_1,\dots, C_p>$ we have $p=L+K \leq \frac{1}{2}[M(M-1)+(q-M)(q-M-1)]$.  As $2\leq M \leq q-2$ we see that $p\leq \frac{1}{2}[2+(q-2)(q-3)] = \frac{1}{2}(q-2)(q-3)+1$.  This proves the lemma.
\end{proof}

\begin{lemma} If $\N$ is of type $(2,q)$, with no Euclidean de Rham factor, and is a non-Einstein nilradical, then $\N$ is indecomposable.
\end{lemma}

\begin{proof} Suppose that $\N$ decomposes as $\N_1\oplus \N_2$, a direct sum of ideals. As $\N$ has no Euclidean de Rham factor, neither $\N_1$ nor $\N_2$ is abelian, and hence, both have one dimensional centers.  However, there is only one two-step nilalgebra (up to isomorphism) with one dimensional center and it is an Einstein nilradical.  Applying Theorem \ref{thm: Einstein and direct sums} we would then have $\N$ is an Einstein nilradical, which contradicts our hypothesis.
\end{proof}

\providecommand{\bysame}{\leavevmode\hbox to3em{\hrulefill}\thinspace}
\providecommand{\MR}{\relax\ifhmode\unskip\space\fi MR }
\providecommand{\MRhref}[2]{%
  \href{http://www.ams.org/mathscinet-getitem?mr=#1}{#2}
}
\providecommand{\href}[2]{#2}

\end{document}